\documentclass[11pt]{amsart}

\usepackage{amsmath}
\usepackage{amssymb,mathtools,amsthm,latexsym}

\calclayout

\title{Sobolev extensions over Cantor-cuspidal graphs}
\author{Pekka Koskela and Zheng Zhu}

\address{Pekka Koskela\\
Department of Mathematics and Statistics\\
University of Jyv\"askyl\"a, P.O. Box 35 (MaD),
FI-40014, Jyv\"askyl\"a, Finland}
\email{\tt pekka.j.koskela@jyu.fi}

\address{Zheng Zhu\\
School of Mathematical science\\
         Beihang University\\
        Changping District Shahe Higher Education Park South Third Street No. 9\\
        Beijing 102206,\\
        P. R. China}
\email{\tt zhzhu@buaa.edu.cn}

\usepackage{stackrel}

\usepackage{color}

\numberwithin{equation}{section}
\long\def\colred#1\endred{{\color{red}#1}}
\long\def\colgreen#1\endgreen{{\color{green}#1}}
\long\def\colmagenta#1\endmagenta{{\color{magenta}#1}}
\long\def\colblue#1\endblue{{\color{blue}#1}}
\long\def\colyellow#1\endyellow{{\color{yellow}#1}}

\usepackage{tikz}
\usetikzlibrary{matrix,arrows}

\theoremstyle{plain}

\theoremstyle{remark}

\theoremstyle{definition}

\newtheorem{defn}[equation]{Definition}

\newtheorem*{question*}{Question}

\usepackage{graphicx}


\thanks{The first author has been supported  by the Academy of Finland via Centre of Excellence in Analysis and Dynamics Research (Project \#323960). The second author has been supported by the NSFC grant no. 12301111.}

\newcounter{prob}
\def\rr{{\mathbb R}}
\def\rn{{{\rr}^n}}

\def\fz{\infty}

\def\loc{{\mathop\mathrm{\,loc\,}}}

\def\boz{{\Omega}}

\def\ls{\lesssim}

\def\mr{\mathcal R}

\def\bint{{\ifinner\rlap{\bf\kern.25em--}
\int\else\rlap{\bf\kern.45em--}\int\fi}\ignorespaces}

\def\bbint{{\ifinner\rlap{\bf\kern.25em--}
\hspace{0.078cm}\int\else\rlap{\bf\kern.45em--}\int\fi}\ignorespaces}

\def\r{\right}
\def\lf{\left}

\newcommand{\R}{\ensuremath{\mathbb{R}}}

\def\XXint#1#2#3{{\setbox0=\hbox{$#1{#2#3}{\int}$ }
\vcenter{\hbox{$#2#3$ }}\kern-.58\wd0}}

\newtheorem{thm}{Theorem}[section]
\newtheorem{lem}{Lemma}[section]
\newtheorem{prop}{Proposition}[section]

\numberwithin{equation}{section}
\def\vint_#1{\mathchoice%
        {\mathop{\kern 0.2em\vrule width 0.6em height 0.69678ex depth -0.58065ex
                \kern -0.8em \intop}\nolimits_{\kern -0.4em#1}}%
        {\mathop{\kern 0.1em\vrule width 0.5em height 0.69678ex depth -0.60387ex
                \kern -0.6em \intop}\nolimits_{#1}}%
        {\mathop{\kern 0.1em\vrule width 0.5em height 0.69678ex
            depth -0.60387ex
                \kern -0.6em \intop}\nolimits_{#1}}%
        {\mathop{\kern 0.1em\vrule width 0.5em height 0.69678ex depth -0.60387ex
                \kern -0.6em \intop}\nolimits_{#1}}}
\def\vintslides_#1{\mathchoice%
        {\mathop{\kern 0.1em\vrule width 0.5em height 0.697ex depth -0.581ex
                \kern -0.6em \intop}\nolimits_{\kern -0.4em#1}}%
        {\mathop{\kern 0.1em\vrule width 0.3em height 0.697ex depth -0.604ex
                \kern -0.4em \intop}\nolimits_{#1}}%
        {\mathop{\kern 0.1em\vrule width 0.3em height 0.697ex depth -0.604ex
                \kern -0.4em \intop}\nolimits_{#1}}%
        {\mathop{\kern 0.1em\vrule width 0.3em height 0.697ex depth -0.604ex
                \kern -0.4em \intop}\nolimits_{#1}}}

\begin{document}

\maketitle

\begin{abstract}
For a continuous function $f:\rr\to\rr$, define the corresponding graph by setting
$$\Gamma_f:=\lf\{(x_1, f(x_1)): x_1\in\rr\r\}.$$ 
In this paper, we study the Sobolev extension property for the upper and lower domains over the graph $\Gamma_{\psi^\alpha_c}$ for $\psi^\alpha_c(x_1)=d(x_1, \mathcal C)^\alpha$, where $\mathcal C$ is the classical ternary Cantor set in the unit interval and $\alpha\in(0, 1)$.  
\end{abstract}

\section{Introduction}
Let $1\leq q\leq p\leq\fz$. If $\psi$ is a Lipschitz function on $\mathbb R,$ then the upper and lower domains determined by the graph of $\psi$ are extension domains for all the first order Sobolev spaces 
$W^{1,p}$ by the works of Calder\'on and Stein \cite{stein}: if $\Omega$ is either one of these domains, then there is a linear bounded extension operator from $W^{1,p}(\Omega)$ into $W^{1,p}(\mathbb R^2).$
In fact, it is not hard to check that one can obtain such an extension operator via a bi-Lipschitz reflection that switches the upper and lower domains: the bi-Lipschitz map $f(x,y)=(x,y-\psi(x))$ maps the two domains onto half-planes and one can use the usual reflection with respect to the first coordinate axes, modulo $f$ and its inverse. 

If $\psi$ is H\"older-continuous, one cannot necessarily extend from $W^{1,p}(\Omega)$
to $W^{1,p}(\mathbb R^2)$ since the graph of $\psi$ can contain cusps. In this case, under the correct assumptions on $p,q$ and the H\"older-exponent, one can nevertheless extend from $W^{1,p}(\Omega)$ to $W^{1,q}(\mathbb R^2),$  see \cite{Fain}, in the sense that the extension belongs to $W^{1,q}_{\loc}(\R^2)$ with 
$$||Eu||_{W^{1,q}(\mathbb R^2\setminus \Omega)}\le C||u||_{W^{1,p}(\Omega)}.$$
More precisely, if $1/2<\alpha<1$ and $p>\frac{\alpha}{2\alpha-1}$ then any $1\le q<\frac{\alpha p}{\alpha+(1-\alpha)p}$ can be obtained. 
For a single cusp, the sharp exponents for the interior are $p>\frac {1+\alpha}{2\alpha}$ and $q<\frac {2\alpha} {(1+\alpha)p}$ and for the exterior  $p>1$ and $q<\frac {(1+\alpha)p}{2\alpha+(1-\alpha)p}$ or $p=1=q$. 
This follows from more general work of
 Gol'dshtein and Sitnikov \cite{GS} who showed that for certain cusp-like domains with H\"older boundaries one can obtain an extension via a reflection which in this case is not anymore bi-Lipschitz. For an exposition and more historical references for the theory of Sobolev spaces on non-smooth domains, we refer the reader to \cite{Mazya}.

In this paper we consider the above extension problem in a model case in oder to gain better insight to the problem. Towards our model case, let $\mathcal C\subset [0,1]\times \{0\}\subset \mathbb R^2$ be the standard ternary Cantor set, obtained by removing a sequence of `centrally located" open intervals from $[0,1]\times \{0\}$: at stage $j$ we have $2^j$ closed intervals, each of length $3^{-j},$ from the middle of each we remove an open interval of length $3^{-j-1}.$
 For $\alpha\in (0, 1)$, we define $\psi^\alpha_c:\rr\to\rr$ by setting 
 \begin{equation}\label{eq:function}
 \psi^\alpha_c(x_1):=\begin{cases}
 d(x_1, \mathcal C)^\alpha, \ \ &\ {\rm if}\ x_1\in(0, 1)\\
 0, \ \ &\ {\rm if}\ x_1\in\rr\setminus(0, 1).
 \end{cases}
 \end{equation}
 Then $\psi^\alpha_c$ is H\"older continuous with exponent $\alpha.$ Define the corresponding graph by setting
\[\Gamma_{\psi^\alpha_c}:=\{(x_1, \psi^\alpha_c(x_1)): x_1\in\rr\},\]  
and let $\Omega^+_{\psi^\alpha_c}$  be the  domain above the ``Cantor-cuspidal" graph $\Gamma_{\psi^\alpha_c}:$ 
\[\Omega^+_{\psi^\alpha_c}:=\lf\{(x_1,x_2)\in\rr^2: x_1\in\rr\ {\rm and}\ x_2>\psi^\alpha_c(x_1)\r\},\]
 and $\Omega^-_{\psi^\alpha_c}$ be the domain below the ``Cantor-cuspidal" graph $\Gamma_{\psi^\alpha_c}:$ 
 \[\Omega^-_{\psi^\alpha_c}:=\lf\{(x_1, x_2)\in\rr^2: x_1\in\rr\ {\rm and}\ x_2>\psi^\alpha_c(x_1)\r\}.\]
See Figure $1$ below. Notice that 
\[\limsup_{t\to 0^-}\frac { \psi^\alpha_c(x_1+t)}{t^\alpha}>0\ {\rm and}\  
 \limsup_{t\to 0^+}\frac { \psi^\alpha_c(x_1+t)}{t^\alpha}>0\]
  for all $x_1\in \mathcal C\setminus \{0,1\}.$ Hence the common boundary of both of these domains is `cusp-like" in a set of  Hausdorff dimension $\frac
 {\log 2}{\log 3}.$
 
Towards our results, we call a homeomorphism $\mr:\rr^2\rightarrow \rr^2$ a reflection over $\Gamma_{\psi^\alpha_c}$ if $\mr(\Omega^+_{\psi^\alpha_c})=\Omega^-_{\psi^\alpha_c}$ and $\mr(x)=x$ for all
$x\in \Gamma_{\psi^\alpha_c}.$ Let $\Omega$ be one of the domains $ \Omega^+_{\psi^\alpha_c},\Omega^-_{\psi^\alpha_c}.$ We say that a reflection $\mr$ over  $\Gamma_{\psi^\alpha_c}$ induces a bounded linear extension operator from $W^{1,p}(\Omega)$ to $W_{\rm loc}^{1,q}(\rr^2)$ if for every $u\in W^{1,p}(\Omega),$ 
the function $v$ defined by setting
$v=u$ on $\Omega$ and $v=u\circ \mr$ on $\rr^2\setminus\overline{\boz}$ has a representative that belongs to $W_{\rm loc}^{1,q}(\rr^2)$ such that for every bounded open set $U\subset\rr^2$, we have
\begin{equation} \label{uu1}
\|v\|_{W^{1,q}(U)}\leq C\|u\|_{W^{1,p}(\boz)},
\end{equation}
for a positive constant $C$ independent of $u$. In our setting, this conclusion easily implies that one can find an extension $Eu$ with 
$$||Eu||_{W^{1,q}(\rr^2\setminus \Omega)}\le C||u||_{W^{1,p}(\Omega)}.$$

\begin{figure}[htbp]
\centering
\includegraphics[width=0.8\textwidth]
{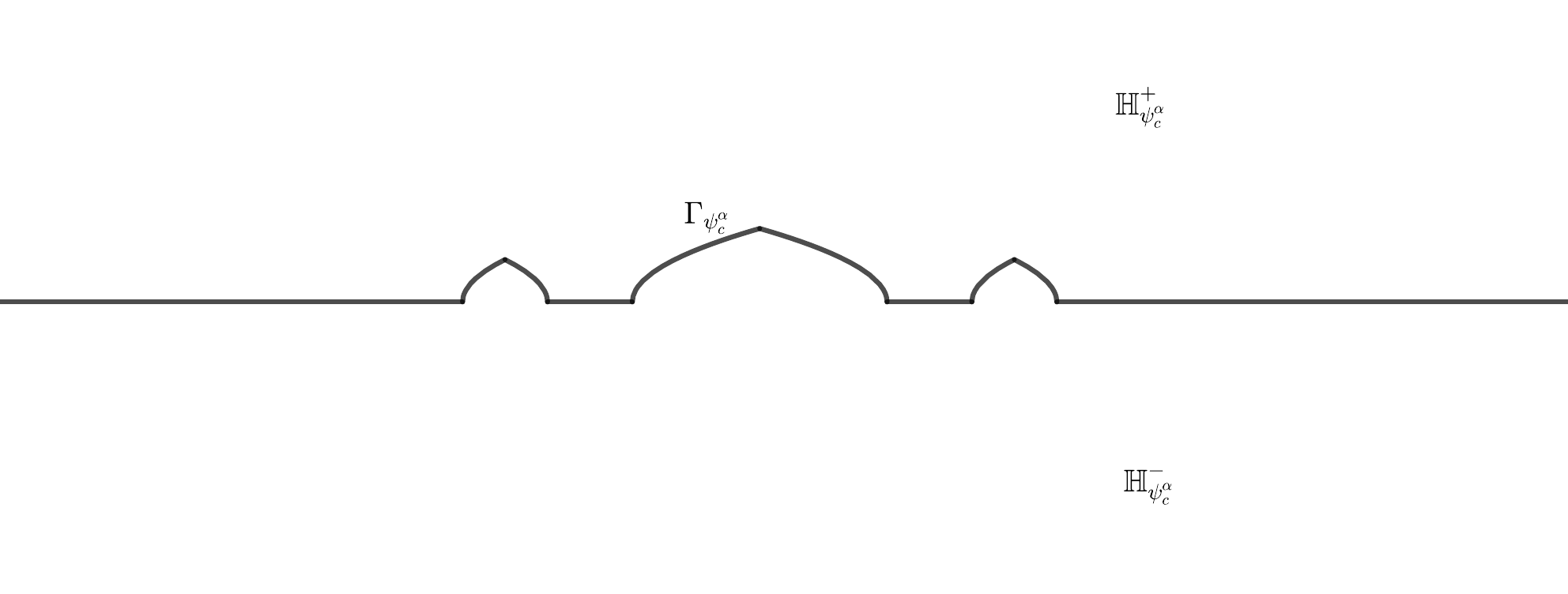}\label{figu2}
\caption{Cantor-cuspidal graph $\Gamma_{\psi^\alpha_c}$}
\end{figure}



The following theorem is the main result of the article. 
 \begin{thm}\label{thm:extension}
Let $\frac{\log 2}{2\log 3}<\alpha<1$. Then there exists a reflection $\mr:{\rr^2}\to{\rr^2}$ over $\Gamma_{\psi^\alpha_c}$ which induces a bounded linear extension operator from $W^{1,p}(\Omega^+_{\psi^\alpha_c})$ to $W_{\rm loc}^{1,q}(\rr^2)$ and a bounded linear extension operator from $W^{1, p}(\Omega^-_{\psi^\alpha_c})$ to $W_{\rm loc}^{1,q}(\rr^2)$, whenever $\frac{(1+\alpha)-\frac{\log 2}{\log 3}}{2\alpha-\frac{\log 2}{\log 3}}<p<\fz$ and $1\leq q<\frac{((1+\alpha)-\frac{\log 2}{\log 3})p}{(1+\alpha)-\frac{\log 2}{\log 3}+(1-\alpha)p}$.
\end{thm}

The following proposition shows the sharpness of Theorem \ref{thm:extension}.
\begin{prop}\label{thm:nonex1}
$(1):$ Let $0<\alpha\leq\frac{\log 2}{2\log 3}$ be fixed. Then for arbitrary $1\leq p\leq \fz$, there exist functions $u_e^-\in W^{1, p}(\Omega^-_{\psi^\alpha_c})$ and $u_e^+\in W^{1, p}(\Omega^+_{\psi^\alpha_c})$ which do not have extensions in the class $W^{1,1}_{\rm loc}(\rr^2)$. 

$(2):$ Let $\frac{\log 2}{2\log 3}<\alpha<1$ be fixed. Then for arbitrary $1\leq p\leq\frac{(1+\alpha)-\frac{\log 2}{\log 3}}{2\alpha-\frac{\log 2}{\log 3}}$, there exist functions $u_e^-\in W^{1,p}(\Omega^-_{\psi_c^s})$ and $u_e^+\in W^{1,p}(\Omega^+_{\psi^\alpha_c})$ which do not have extensions in the class $W^{1, 1}_{\rm loc}(\rr^2)$. 

$(3):$ Let $\frac{\log 2}{2\log 3}<\alpha<1$ be fixed. Then for arbitrary $\frac{(1+\alpha)-\frac{\log 2}{\log 3}}{2\alpha-\frac{\log 2}{\log 3}}<p<\fz$ and $\frac{((1+\alpha)-\frac{\log 2}{\log 3})p}{(1+\alpha)-\frac{\log 2}{\log 3}+(1-\alpha)p}\leq q<\fz$, there exist functions $u_e^-\in W^{1, p}(\Omega^-_{\psi^\alpha_c})$ and $u_e^+\in W^{1, p}(\Omega^+_{\psi^\alpha_c})$ which do not have extensions in the class $W^{1,q}_{\rm loc}(\rr^2)$.
\end{prop}

We would like to know if Theorem \ref{thm:extension} exhibits a general principle: could it be the case that for graphs the Sobolev extension problem is equivalent to the existence of a suitable reflection? For partial results in this direction see \cite{GS,KPZ}.  The symmetry in Theorem \ref{thm:extension} cannot hold in general as follows by the results in \cite{GS}, our reflection has better properties that one would in general expect \cite{HKARMA, nre1}.

We close this introduction with a comment regarding the case $p=1=q.$
A domain $\boz\subset\rn$ is called quasiconvex, if there exists a positive constant $C>1$ such that for every pair of points $x, y\in\boz$, there exists a rectifiable curve $\gamma\subset\boz$ joining $x, y$ with 
\[l(\gamma)\leq C|x-y|.\]
For every $0<\alpha<1$, one can easily see that neither $\Omega^+_{\psi^\alpha_c}$ nor $\Omega^-_{\psi^\alpha_c}$ is quasiconvex. Then the corollary in \cite{KMS} with a bit of work implies that neither $\Omega^+_{\psi^\alpha_c}$ nor $\Omega^-_{\psi^\alpha_c}$ is a Sobolev $(1, 1)$-extension domain. Hence, in the proof to Proposition \ref{thm:nonex1} below, we will only discuss the case for $p>1$.

\section{Preliminaries}
The notation $x=(x_1, x_2)\in\rr^2$ means a point in the Euclidean plane $\rr^2$. Typically $C$ will be a constant that depends on various parameters and may differ even on the same line of inequalities. The notation $A\ls B$ means there exists a finite constant $C$ with $A\le CB$ , and $A\sim_{C} B$ means $\frac{1}{C}A\leq B\leq CA$ for a constant $C>1$. The Euclidean distance between points $x, x'$ in the Euclidean plane $\mathbb R^2$ is denoted by $|x-x'|$. The open disk of radius $r$ centered at $x$ is denoted by $D(x,r)$. $\mathcal H^2(A)$ means the $2$-dimensional Lebesgue measure for a measurable set $A\subset\mathbb R^2$. 

To obtain the classical ternary Cantor set in the unit interval $[0, 1]\subset\rr$, we remove a class of pairwise disjoint open intervals step by step. At the first step, we remove the middle $\frac{1}{3}$-interval $I_1^1:=(\frac{1}{3}, \frac{2}{3})$ from the unit interval $I:=[0, 1]$. At the second step, we remove the middle $\frac{1}{9}$-intervals $I_2^1:=(\frac{1}{9}, \frac{2}{9})$ and $I_2^2:=(\frac{7}{9}, \frac{8}{9})$ from the two intervals of length $\frac{1}{3}$ obtained from the first step. By induction, at the $n$-th step, we remove the middle $\frac{1}{3^n}$-intervals from the $2^{n-1}$ intervals $\{I_n^k\}_{k=1}^{2^{n-1}}$ of length $\frac{1}{3^{n-1}}$ obtained from the $(n-1)$-th step. Finally, we obtain the classical ternary Cantor set
$$\mathcal C:=[0,1]\setminus\bigcup_{n=1}^{\fz}\bigcup_{k=1}^{2^{n-1}}I_n^k.$$
Let us denote the removed open intervals by $I_n^k:=(a_n^k, b_n^k)$. Then, the function $\psi^\alpha_c$ defined in (\ref{eq:function}) can be rewritten as
\begin{equation}\label{psi}
\psi^\alpha_c(x_1):=\left\{\begin{array}{ll}
\min\lf\{|x_1-a_n^k|^\alpha, |x_1-b_n^k|^\alpha\r\},&\ x_1\in I_n^k,\\
0,&\ {\rm elsewhere}.
\end{array}\right.
\end{equation}

Let us give the definition of Sobolev spaces.
\begin{defn}\label{defnsobolev}
Let $\boz\subset\mathbb R^2$ be a domain and $u\in L^{\, 1}_{\rm loc}(\boz)$. A vector function $v\in L^{\, 1}_{\rm loc}(\boz,\mathbb R^2)$ is called a weak derivative of $u$ if 
\begin{equation}
\int_{\boz}\eta(x)v(x)dx=-\int_{\boz}u(x)D\eta(x)dx\nonumber
\end{equation}
holds for every function $\eta\in C_c^\fz(\boz)$. We refer to $v$ by $Du$. For $1\leq p\leq\fz$, we define the Sobolev space by 
\begin{equation}
W^{1,p}(\boz):=\{u\in L^{\, p}(\boz): |Du|\in L^{\, p}(\boz)\}\nonumber
\end{equation}
and we define the norm by  
\begin{equation}
\|u\|_{W^{1,p}(\boz)}:=\lf(\int_\boz|u(x)|^pdx\r)^{\frac{1}{p}}+\lf(\int_\boz|Du(x)|^pdx\r)^{\frac{1}{p}}.\nonumber
\end{equation}
If for every bounded open subset $U\subset\boz$ with $\overline{U}\subset\boz$, $u\in W^{1, p}(U)$, then we say $u\in W^{1, p}_{\rm loc}(\boz)$. 
\end{defn}


Although the Cantor-cuspidal graph $\Gamma_{\psi^\alpha_c}$ has a plethora of singularities, the corresponding upper and lower domains $\Omega^+_{\psi^\alpha_c}$ and $\Omega^-_{\psi^\alpha_c}$ still enjoy some nice geometric properties. For example, they satisfy the following so-called segment condition.
\begin{defn}\label{defn:segment}
We say that a domain $\boz\subset\rr^2$ satisfies the segment condition if every $x\in\partial\boz$ has a neighborhood $U_x$ and a nonzero vector $y_x$ such that if $z\in\overline{\boz}\cap U_x$, then $z+ty_x\in\boz$ for $0<t<1$.
\end{defn}
 For domains which satisfy the segment condition, we have the following density result. See \cite[Theorem 3.22]{adams}.
\begin{lem}\label{lem:density}
If the domain $\boz\subset\rr^2$ satisfies the segment condition, then the set of restrictions to $\boz$ of functions in $C_c^\fz(\rr^2)$ is dense in $W^{1,p}(\boz)$ for $1\leq p<\fz$. In short, $C_c^\fz(\rr^2)\cap W^{1,p}(\boz)$ is dense in $W^{1,p}(\boz)$ for $1\leq p<\fz$.
\end{lem}

By combining the theorems in \cite{Ukhlov, Vodop1, Vodop2, Vodop3}, also see \cite{Ukhlov2, Ukhlov3}, we obtain the following lemma.
\begin{lem}\label{lem:pQc}
Let $1\leq q<p<\fz$. Suppose that $f:\boz\to\boz'$ is a homeomorphism in the class $W^{1, 1}_{\rm loc}(\boz, \boz').$ Then the following assertions are equivalent:\\
$(1):$ for every locally Lipschitz function $u$ defined on $\boz'$, the inequality 
\[\lf(\int_\boz|D(u\circ f)(x)|^qdx\r)^{\frac{1}{q}}\leq C\lf(\int_{\boz'}|Du(x)|^pdx\r)^{\frac{1}{p}}\]
holds for a positive constant $C$ independent of $u$;\\
$(2):$ \[\int_{\boz}\frac{|Df(x)|^{\frac{pq}{p-q}}}{|J_f(x)|^{\frac{q}{p-q}}}dx<\fz.\]
\end{lem}

\section{A reflection over $\Gamma_{\psi^\alpha_c}$}
In this section, we always assume $\frac{\log 2}{2\log 3}<\alpha<1$. To begin, we define a class of sets $\{\mathsf R_{n, k}; k=1, ...,2^{n-1} \}_{n=1}^{\fz}$ by setting 
\begin{equation}\label{equa:rec}
\mathsf R_{n, k}:=\lf\{(x_1, x_2)\in\mathbb R^2; x_1\in I_n^k, x_2\in\lf(-2\psi^\alpha_c(x_1), 2\psi^\alpha_c(x_1)\r)\r\}.
\end{equation}
We also define $\mathsf R^+_{n, k}$ and $\mathsf R^-_{n, k}$ to be the corresponding upper and lower parts of $\mathsf R_{n, k}$ with respect to $\Gamma_{\psi^\alpha_c}$ 
by setting
\begin{equation}\label{upperrec}
\mathsf R^+_{n, k}:=\lf\{(x_1, x_2)\in\mathbb R^2; x_1\in I_n^k, x_2\in\lf(\psi^\alpha_c(x_1), 2\psi^\alpha_c(x_1)\r)\r\}
\end{equation}
and
\begin{equation}\label{lowerrec}
\mathsf R^-_{n, k}:=\lf\{(x_1, x_2)\in\mathbb R^2; x_1\in I_n^k, x_2\in\lf(-2\psi^\alpha_c(x_1), \psi^\alpha_c(x_1)\r)\r\}.
\end{equation}
 Fix $x_1\in I_n^k$ for some $n\in\mathbb N$ and $k\in\{1, 2, \cdots, 2^{n-1}\}.$ Then
\begin{equation}\label{uppline}
\mathsf S^+_{x_1}:=\lf\{(x_1, x_2)\in\mathbb R^2; \psi^\alpha_c(x_1)<x_2<2\psi^\alpha_c(x_1)\r\}
\end{equation}
is a vertical line segment in $\mathsf R^+_{n ,k}$ and 
\begin{equation}\label{lowline}
\mathsf S^-_{x_1}:=\lf\{(x_1, x_2)\in\mathbb R^2; -2\psi^\alpha_c(x_1)<x_2< \psi^\alpha_c(x_1)\r\}
\end{equation}
is a vertical line segment in $\mathsf R^-_{n, k}$. 

\begin{figure}[htbp]
\centering
\includegraphics[width=0.8\textwidth]
{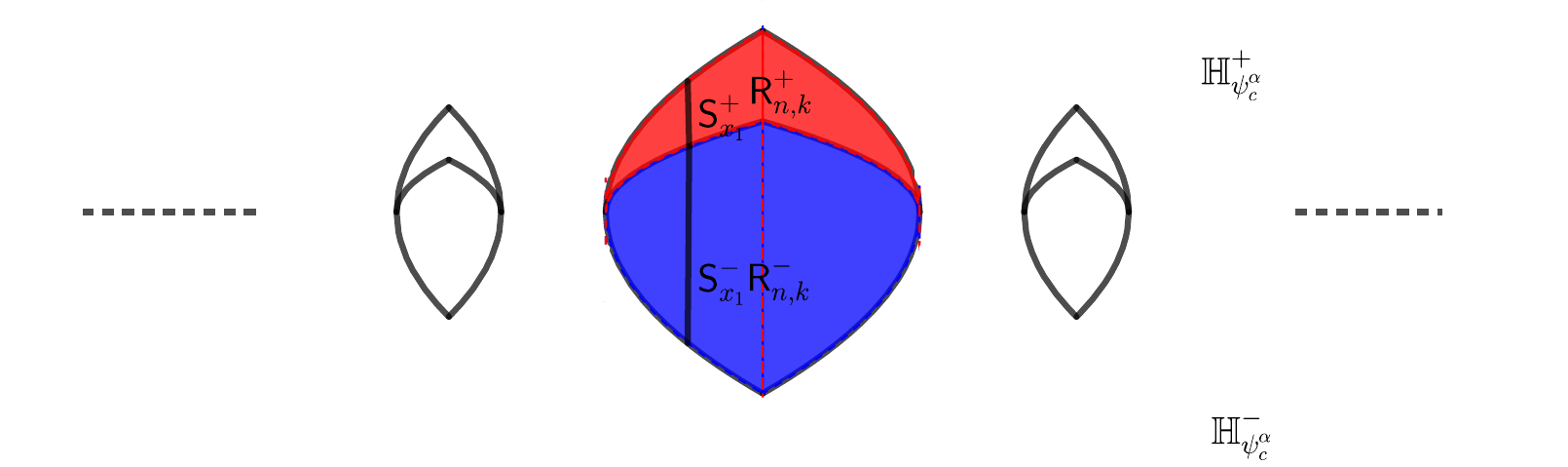}\label{figu2}
\caption{Reflection $\mr$}
\end{figure}

Now, we are ready to define our reflection $\mr:\rr^2\to\rr^2$ over $\Gamma_{\psi^\alpha_c}$. Our reflection will map the segment $\mathsf{S}^+_{x_1}$ onto the segment $\mathsf S^-_{x_1}$ affinely for every $x_1\in I_n^k$. To be precise, we define the reflection $\mr$ on $\Omega^+_{\psi^\alpha_c}$ by setting 
\begin{equation}\label{equa:reflec}
\mr(x)=\left\{\begin{array}{ll}
\lf(x_1, -3x_2+4\psi^\alpha_c(x_1)\r), &\  x\in \mathsf R^+_{n, k},\\
(x_1, -x_2), &\ {\rm elsewhere}.
\end{array}\right.
\end{equation}
where $x=(x_1, x_2)$. It is easy to check that $\mr$ maps $\mathsf R^+_{n, k}$ onto $\mathsf R^-_{n, k}$, for every $n\in\mathbb N$ and $k\in\{1, 2, \cdots, 2^{n-1}\}$. On $\Omega^-_{n, k}$, we simply let $\mr$ be the inverse of (\ref{equa:reflec}). Since $\mr(\Omega^+_{\psi^\alpha_c})=\Omega^-_{\psi^\alpha_c}$, $\mr(\Omega^-_{\psi^\alpha_c})=\Omega^+_{\psi^\alpha_c}$ and $\mr(x)=x$ for every $x\in\Gamma_{\psi^\alpha_c}$, $\mr$ is a reflection over $\Gamma_{\psi^\alpha_c}$. For every $x_1\in I_n^k$ with $n\in\mathbb N$ and $k\in\{1,2,\cdots, 2^{n-1}\}$, the reflection $\mr$ maps $\mathsf S^+_{x_1}$ onto $\mathsf S^-_{x_1}$ affinely. 

By a simple computation, at every point $x\in\mathsf R^+_{n, k}$ for some $n\in\mathbb N$ and $k\in\{1, 2, \cdots, 2^{n-1}\}$, the differential matrix $D\mr(x)$ is 
\begin{equation}\label{equa:matrix}
D\mr(x)
=
\left(
 \begin{array}{ccc}
1 &~~ 0\\

 4(\psi^\alpha_c)'(x_1)&~~ -3 \\
\end{array}
\right).
\end{equation}
By another simple computation, there exists a positive constant $C>0$ such that for every $x\in\mathsf R^+_{n,k}$, we have
\begin{equation}
|D\mr(x)|\leq C|(\psi^\alpha_c)'(x_1)|\leq Cd(x_1, \mathcal C)^{\alpha-1}.
\end{equation}
and 
\begin{equation}\label{eq:jaco1}
|J_\mr(x)|=3.
\end{equation}
Next, let us estimate the analog of $(2)$ of Lemma \ref{lem:pQc}  for $\mr$: 
\begin{eqnarray}\label{equa:finite}
C_+&:=&\int_{\cup_{n=1}^\fz\cup_{k=1}^{2^{n-1}}\mathsf R^+_{n, k}}\frac{|D\mr(x)|^{\frac{pq}{p-q}}}{|J_\mr(x)|^{\frac{q}{p-q}}}dx\\
     &\leq&C\sum_{n=1}^{\fz}\sum_{k=1}^{2^{n-1}}\int_{a_n^k}^{a_n^k+\frac{1}{2\cdot3^n}}\int_{\psi^\alpha_c(x_1)}^{2\psi^\alpha_c(x_1)}(x_1-a_n^k)^{\frac{(\alpha-1)pq}{p-q}}dx_2dx_1\nonumber\\
                                  &\leq&C\sum_{n=1}^{\fz}2^n\int_{0}^{\frac{1}{2\cdot3^n}}x_1^{\alpha+\frac{(\alpha-1)pq}{p-q}}dx_1\nonumber\\
                                  &\leq&C\sum_{n=1}^{\fz}2^n\lf(\frac{1}{3^n}\r)^{1+\alpha+\frac{(\alpha-1)pq}{p-q}}<\infty,\nonumber
\end{eqnarray}
whenever $\frac{(1+\alpha)-\frac{\log 2}{\log 3}}{2\alpha-\frac{\log 2}{\log 3}}<p<\fz$ and $1\leq q<\frac{((1+\alpha)-\frac{\log 2}{\log 3})p}{(1+\alpha)-\frac{\log 2}{\log 3}+(1-\alpha)p}$.  

By a simple computation, we can rewrite the reflection $\mr$ on $\Omega^-_{\psi^\alpha_c}$ as
\begin{equation}\label{equa:reflec1}
\mr(x)=\left\{\begin{array}{ll}
\lf(x_1, \frac{-1}{3}x_2+\frac{4}{3}\psi^\alpha_c(x_1)\r), &\  x\in \mathsf R^-_{n, k},\\
(x_1, -x_2), &\ {\rm elsewhere}.
\end{array}\right.
\end{equation}
where $x=(x_1, x_2)$. At every point $x\in\mathsf R^-_{n, k}$ for $n\in\mathbb N$ and $k\in\{1, 2, \cdots, 2^{n-1}\}$, the differential matrix $D\mr(x)$ is 
\begin{equation}\label{equa:matrix1}
D\mr(x)
=
\left(
 \begin{array}{ccc}
1 &~~ \frac{4}{3}(\psi^\alpha_c)'(x_1)\\

0 &~~ \frac{-1}{3} \\
\end{array}
\right).
\end{equation}
There exists a positive constant $C>1$ such that 
\begin{equation}
|D\mr(x)|\leq C|(\psi^\alpha_c)'(x_1)|\leq Cd(x_1, \mathcal C)^{\alpha-1}
\end{equation}
and 
\begin{equation}\label{eq:jaco2}
|J_\mr(x)|=\frac{1}{3}.
\end{equation}
Hence, we have
\begin{eqnarray}\label{equa:Finite}
C_-&:=&\int_{\cup_{n=1}^\fz\cup_{k=1}^{2^{n-1}}\mathsf R^-_{n, k}}\frac{|D\mr(x)|^{\frac{pq}{p-q}}}{|J_\mr(x)|^{\frac{q}{p-q}}}dx\\
&\leq&C\sum_{n=1}^{\fz}\sum_{k=1}^{2^{n-1}}\int_{a_n^k}^{a_n^k+\frac{1}{2\cdot3^n}}\int_{-2\psi^\alpha_c(x_1)}^{\psi^\alpha_c(x_1)}(x_1-a_n^k)^{\frac{(\alpha-1)pq}{p-q)}}dx_2dx_1\nonumber\\
&\leq&C\sum_{n=1}^{\fz}2^n\int_{0}^{\frac{1}{2\cdot3^n}}x_1^{\alpha+\frac{(1-s)pq}{s(p-q)}}dx_1\nonumber\\
                &\leq&C\sum_{n=1}^{\fz}2^n\lf(\frac{1}{3^n}\r)^{1+\alpha+\frac{(\alpha-1)pq}{p-q}}<\infty,\nonumber
\end{eqnarray}
whenever $\frac{(1+\alpha)-\frac{\log 2}{\log 3}}{2\alpha-\frac{\log 2}{\log 3}}<p<\fz$ and $1\leq q<\frac{((1+\alpha)-\frac{\log 2}{\log 3})p}{(1+\alpha)-\frac{\log 2}{\log 3}+(1-\alpha)p}$. 

Finally, it is easy to see that the reflection $\mr$ is bi-Lipschitz on every open set $U\subset\rr^2$ with $d(U, \Gamma_{\psi^\alpha_c})>0$. 

\section{Sobolev extendability for $\Omega^+_{\psi^\alpha_c}$}
\subsection{Extension from $W^{1,p}(\Omega^+_{\psi^\alpha_c})$ to $W^{1,q}_{\rm loc}(\rr^2)$}
\begin{thm}\label{thm:E1}
Let $\Gamma_{\psi^\alpha_c}\subset\rr^2$ be a Cantor-cuspidal graph with $\frac{\log 2}{2\log 3}<\alpha<1$. The reflection $\mr:\rr^2\to\rr^2$ over $\Gamma_{\psi^\alpha_c}$ defined in (\ref{equa:reflec}) and (\ref{equa:reflec1}) induces a bounded linear extension operator from $W^{1,p}(\Omega^+_{\psi^\alpha_c})$ to $W^{1,q}_{\rm loc}(\rr^2)$, whenever $\frac{(1+\alpha)-\frac{\log 2}{\log 3}}{2\alpha-\frac{\log 2}{\log 3}}<p<\fz$ and $1\leq q<\frac{((1+\alpha)-\frac{\log 2}{\log 3})p}{(1+\alpha)-\frac{\log 2}{\log 3}+(1-\alpha)p}$.  
\end{thm}
\begin{proof}
Let $\frac{(1+\alpha)-\frac{\log 2}{\log 3}}{2\alpha-\frac{\log 2}{\log 3}}<p<\fz$ and $1\leq q<\frac{((1+\alpha)-\frac{\log 2}{\log 3})p}{(1+\alpha)-\frac{\log 2}{\log 3}+(1-\alpha)p}$ be fixed. Since $\Omega^+_{\psi^\alpha_c}$ satisfies the segment condition defined in Definition \ref{defn:segment}, by Lemma \ref{lem:density}, $C_c^\fz(\rr^2)\cap W^{1,p}(\Omega^+_{\psi^\alpha_c})$ is dense in $W^{1,p}(\Omega^+_{\psi^\alpha_c})$. Let $u\in C_c^\fz(\rr^2)\cap W^{1,p}(\Omega^+_{\psi^\alpha_c})$ be arbitrary. We define a function $\tilde E_\mr(u)$ on $\rr^2$ by setting 
\begin{equation}\label{eq:exten1}
\tilde E_{\mr}(u)(x):=\left\{\begin{array}{ll}
u(\mr(x)),&\ {\rm for}\ x\in \rr^2\setminus\overline{\Omega^+_{\psi^\alpha_c}},\\

u(x),&\ {\rm for}\ x\in \overline{\Omega^+_{\psi^\alpha_c}},
\end{array}\right.
\end{equation}
Since $u\in C_c^\fz(\rr^2)\cap W^{1, p}(\Omega^+_{\psi^\alpha_c})$, $\tilde E_\mr(u)$ is continuous on $\rr^2$. Also since $\mr$ is bi-Lipschitz on every open set $U$ with $d(U, \Gamma_{\psi^\alpha_c})>0$, the function $\tilde E_\mr(u)$ is locally Lipschitz on $\rr^2\setminus\Gamma_{\psi^\alpha_c}$. Hence, the weak derivative $D\tilde E_\mr(u)$ is well-defined on $\rr^2\setminus\Gamma_{\psi^\alpha_c}$. 

Let $U\subset\rr^2$ be an arbitrary bounded open set. We define 
$$\tilde U:=\lf(U\cup\mr(U)\r)\cap\Omega^+_{\psi^\alpha_c}.$$
We will show that $\tilde E_\mr(u)\in W^{1,q}_{\rm loc}(\rr^2)$ with 
$$\|\tilde E_\mr(u)\|_{W^{1,q}(U)}\leq C\|u\|_{W^{1,p}(\tilde U)}\leq C\|u\|_{W^{1,p}(\Omega^+_{\psi^\alpha_c})}.$$
Here the constant $C$ may depend on the open set $U$ but must be independent of the function $u$. Since $U\subset\rr^2$ is an arbitrary bounded open set, it suffices to prove inequalities
\begin{equation}\label{eq:norm}
\lf(\int_{U}|\tilde E_\mr(u)(x)|^qdx\r)^{\frac{1}{q}}\leq C\lf(\int_{\tilde U}|u(x)|^pdx\r)^{\frac{1}{p}}
\end{equation}
and
\begin{equation}\label{equa:gnorm}
\lf(\int_{U}|D\tilde E_\mr(u)(x)|^qdx\r)^{\frac{1}{q}}\leq C\lf(\int_{\tilde U}|Du(x)|^pdx\r)^{\frac{1}{p}}.
\end{equation}
Define 
$$U^+:=U\cap\Omega^+_{\psi^\alpha_c}\ {\rm and}\ U^-:=U\cap\Omega^-_{\psi^\alpha_c}.$$
Then $\tilde U= U^+\cup\mr(U^-)$. Since $\mathcal H^2(\Gamma_{\psi^\alpha_c})=0$, we have 
\begin{equation}\label{eq:E1}
\int_U|\tilde E_\mr(u)(x)|^qdx=\int_{U^+}|\tilde E_\mr(u)(x)|^qdx+\int_{U^-}|\tilde E_\mr(u)(x)|^qdx
\end{equation}
and
\begin{equation}\label{eq:E2}
\int_U|D\tilde E_\mr(u)(x)|^qdx=\int_{U^+}|D\tilde E_\mr(u)(x)|^qdx+\int_{U^-}|D\tilde E_\mr(u)(x)|^qdx.
\end{equation}
By the definition of $\tilde E_\mr(u)$ in  (\ref{eq:exten1}), the H\"older inequality implies
\begin{equation}\label{eq:E3}
\lf(\int_{U^+}|\tilde E_\mr(u)(x)|^qdx\r)^{\frac{1}{q}}\leq C\lf(\int_{U^+}|u(x)|^pdx\r)^{\frac{1}{p}}
\end{equation}
and 
\begin{equation}\label{eq:E4}
\lf(\int_{U^+}|D\tilde E_\mr(u)(x)|^qdx\r)^{\frac{1}{q}}\leq C\lf(\int_{U^+}|D\tilde E_\mr(u)(x)|^pdx\r)^{\frac{1}{p}}.
\end{equation}
By (\ref{eq:jaco1}), the H\"older inequality and a change of variables, we have 
\begin{eqnarray}\label{eq:E5}
\int_{U^-}|\tilde E_\mr(u)(x)|^qdx&\leq&\lf(\int_{U^-}|u(\mr(x))|^p|J_\mr(x)|dx\r)^{\frac{q}{p}}\\
                                         & &\cdot\lf(\int_{U^-}\frac{1}{|J_\mr(x)|^{\frac{q}{p-q}}}dx\r)^{\frac{p-q}{p}}\nonumber\\
                                         &\leq&C\lf(\int_{\mr(U^-)}|u(x)|^pdx\r)^{\frac{q}{p}}.\nonumber
\end{eqnarray}
By combining (\ref{eq:E3}) and (\ref{eq:E5}), we obtain the inequality (\ref{eq:norm}). In order to prove the inequality (\ref{equa:gnorm}), it suffices to show 
\begin{equation}\label{eq:E6}
\lf(\int_{U^-}|D\tilde E_\mr(u)(x)|^qdx\r)^{\frac{1}{q}}\leq C\lf(\int_{\mr(U^-)}|Du(x)|^pdx\r)^{\frac{1}{p}}.
\end{equation}
By Lemma \ref{lem:pQc}, it suffices to show that
$$\int_{U^-}\frac{|D\mr(x)|^{\frac{pq}{p-q}}}{|J_\mr(x)|^{\frac{q}{p-q}}}dx<\fz.$$
By a similar computation as in the previous section, we have 
\begin{eqnarray}
\int_{U^-}\frac{|D\mr(x)|^{\frac{pq}{p-q}}}{|J_\mr(x)|^{\frac{q}{p-q}}}dx&=&\int_{U^-\cap\lf(\bigcup_{n, k}\mathsf R^-_{n, k}\r)}\frac{|D\mr(x)|^{\frac{pq}{p-q}}}{|J_\mr(x)|^{\frac{q}{p-q}}}dx\nonumber\\
   & &+\int_{U\setminus\lf(\bigcup_{n, k}\mathsf R^-_{n, k}\r)}\frac{|D\mr(x)|^{\frac{pq}{p-q}}}{|J_\mr(x)|^{\frac{q}{p-q}}}dx\nonumber\\
&\leq&C_-+\mathcal H^2(U)<\fz,\nonumber
\end{eqnarray}
whenever $\frac{(1+\alpha)-\frac{\log 2}{\log 3}}{2\alpha-\frac{\log 2}{\log 3}}<p<\fz$ and $1\leq q<\frac{((1+\alpha)-\frac{\log 2}{\log 3})p}{(1+\alpha)-\frac{\log 2}{\log 3}+(1-\alpha)p}$. By combining (\ref{eq:E4}) and (\ref{eq:E6}), we obtain the inequality (\ref{equa:gnorm}). Hence, the reflection $\mathcal R$ defined in (\ref{equa:reflec}) induces a bounded linear extension operator from $C_c^\fz(\rr^2)\cap W^{1,p}(\Omega^+_{\psi^\alpha_c})$ to $W^{1,q}_{\rm loc}(\rr^2)$. For arbitrary $u\in W^{1,p}(\Omega^+_{\psi^\alpha_c})$, there exists a sequence of function $\{u_m\}_{m=1}^\fz\subset C_c^\fz(\rr^2)\cap W^{1,p}(\Omega^+_{\psi^\alpha_c})$ with 
$$\lim_{m\to\fz}\|u_m-u\|_{W^{1,p}(\Omega^+_{\psi^\alpha_c})}=0$$
and 
$$\lim_{m\to\fz}u_m(x)=u(x)\ {\rm for\ almost\ every}\ x\in\Omega^+_{\psi^\alpha_c}.$$
By combining (\ref{eq:norm}) and (\ref{equa:gnorm}), we obtain 
\begin{equation}\label{eq:E7}
\|\tilde E_\mr(u_m)\|_{W^{1,q}(U)}\leq C\|u_m\|_{W^{1,p}(\tilde U)}\leq C\|u_m\|_{W^{1,p}(\Omega^+_{\psi^\alpha_c})}.
\end{equation}
It implies that $\{\tilde E_\mr(u_m)\big|_U\}_{m=1}^\fz$ is a Cauchy sequence in the Sobolev space $W^{1,q}(U)$ for arbitrary bounded open set $U\subset\rr^2$. Hence, there exists a subsequence of $\{\tilde E_\mr(u_m)\}$ which converges to a Sobolev function $v_U\in W^{1,q}(U)$ point-wise almost everywhere. By covering $\rr^2$ with countably many bounded open sets, there exists a subsequence of $\{\tilde E_\mr(u_m)\}$ which converges to a function $v\in W^{1,q}_{\rm loc}(\rr^2)$ point-wise almost everywhere and $v\big|_U=v_U$ for every bounded open set $U\subset\rr^2$. 
We define $E_\mr(u)$ on $\mathbb C$ by setting 
\begin{equation}\label{equa:exten}
E_{\mr}(u)(x, y):=\left\{\begin{array}{ll}
u(\mr(x)),&\ {\rm for}\ x\in \Omega^+_{\psi_c},\\
0,&\ {\rm for}\ x\in\Gamma_{\psi_c},\\
u(x),&\ {\rm for}\ x\in \Omega^-_{\psi_c}.
\end{array}\right.
\end{equation}
By the definition of $E_\mr(u)$, we have $\lim_{m\to\fz}\widetilde E_\mr(u_m)(x)=E_\mr(u)(x)$ for almost every $x\in\rr^2$. Hence, $E_\mr(u)(x)=v(x)$ for almost every $x\in\rr^2$. It means that $E_\mr(u)\in W^{1,q}_{\rm loc}(\rr^2)$ with 
\begin{equation}\label{eq:E8} 
\|E_\mr(u)\|_{W^{1,q}(U)}\leq C\|u\|_{W^{1, p}(\tilde U)}\leq C\|u\|_{W^{1,p}(\Omega^+_{\psi^\alpha_c})}.
\end{equation}
Since $U\subset\rr^2$ is an arbitrary bounded open set, we proved that the reflection $\mr$ induces a bounded linear extension operator from $W^{1,p}(\Omega^+_{\psi^\alpha_c})$ to $W^{1,q}_{\rm loc}(\rr^2)$, whenever $\frac{(1+\alpha)-\frac{\log 2}{\log 3}}{2\alpha-\frac{\log 2}{\log 3}}<p<\fz$ and $1\leq q<\frac{((1+\alpha)-\frac{\log 2}{\log 3})p}{(1+\alpha)-\frac{\log 2}{\log 3}+(1-\alpha)p}$.
\end{proof}

\subsection{Sharpness of Theorem \ref{thm:E1} for $\Omega^+_{\psi^\alpha_c}$}
 Let $0<\alpha<1$ be fixed. For every $n\in\mathbb N$ and $k=1, ..., 2^{n-1}$, we define a cuspidal domain $\mathcal C_{n, k}$ by setting
\begin{equation}\label{equa:cusp}
\mathcal C_{n,k}:=\lf\{x=(x_1, x_2)\in\mathbb R^2; x_1\in I_{n,k}\ {\rm and}\ 0<x_2<\psi^\alpha_c(x_1)\r\}.
\end{equation}
The reason why we call $\mathcal C_{n, k}$ a cuspidal domain is that the domain $\mathcal C_{n, k}$  has a cuspidal singularity at every end-point of the removed interval $I_{n,k}$.  For $n$-th generation of cuspidal domains $\{\mathcal C_{n, k}\}_{k=1}^{2^{n-1}}$, let $\mathcal C_{n, 1}$ be the left-most one and $\mathcal C_{n, 2^{n-1}}$ be the right-most one. For a cuspidal domain $\mathcal C_{n, k}$ with $1<k<2^{n-1}$ in the $n$-th generation with $n>2$, there exist two cuspidal domains $\mathcal C_{n_1, k_1}$ and $\mathcal C_{n_2, k_2}$ from generations before $n$ which are close to $\mathcal C_{n, k}$. One is on the right-hand side of $\mathcal C_{n,k}$ and the other is on the left-hand side of $\mathcal C_{n, k}$. Let $\mathcal C_{n_1, k_1}$ be in the left-hand side of $\mathcal C_{n, k}$ and $\mathcal C_{n_2,k_2}$ be in the right-hand side of $\mathcal C_{n, k}$. For every removed open interval $I_n^k$, define $q_n^k:=\frac{1}{2}(a_n^k+b_n^k)$ to be the middle point of it. Then we define two open sets $U^l_{n, k}$ and $U^r_{n, k}$ by setting 
$$U^l_{n, k}:=\lf\{x=(x_1, x_2)\in \Omega^+_{\psi^\alpha_c}:q_{n_1}^{k_1}<x_1<q_n^k\ {\rm and}\ \lf(\frac{1}{3}\r)^{\alpha}\lf(\frac{1}{2\cdot 3^n}\r)^{\alpha}<x_2<\lf(\frac{1}{2\cdot 3^n}\r)^{\alpha}\r\}$$ 
and
$$U^r_{n, k}:=\lf\{x=(x_1, x_2)\in \Omega^+_{\psi^\alpha_c}:q_n^k<x_1<q_{n_2}^{k_2}\ {\rm and}\ \lf(\frac{1}{3}\r)^{\alpha}\lf(\frac{1}{2\cdot 3^n}\r)^{\alpha}<x_2<\lf(\frac{1}{2\cdot 3^n}\r)^{\alpha}\r\}.$$
For the left-most generation $n$ cuspidal domain $\mathcal C_{n, 1}$, we define 
\[U^r_{n, 1}:=\lf\{x=(x_1,x_2)\in\Omega^+_{\psi^\alpha_c}: q_n^1<x_1<q_{n-1}^1\ {\rm and}\ \lf(\frac{1}{3}\r)^{\alpha}\lf(\frac{1}{2\cdot3^n}\r)^{\alpha}<x_2<\lf(\frac{1}{2\cdot3^n}\r)^{\alpha}\r\}\]
and 
$$U^l_{n, 1}:=\lf\{x=(x_1,x_2)\in\Omega^+_{\psi^\alpha_c}:-\fz<x_1<q_{n }^1\ {\rm and}\ \lf(\frac{1}{3}\r)^{\alpha}\lf(\frac{1}{2\cdot3^n}\r)^{\alpha}<x_2<\lf(\frac{1}{2\cdot3^n}\r)^{\alpha}\r\}.$$
For the right-most generation $n$ cuspidal domain $\mathcal C_{n, 2^{n-1}}$, we define 
\[U^l_{n, 2^{n-1}}:=\lf\{x=(x_1,x_2)\in\Omega^+_{\psi^\alpha_c}: q^{2^{n-2}}_{n-1}<x_1<q_n^{2^{n-1}}\ {\rm and}\ \lf(\frac{1}{3}\r)^{\alpha}\lf(\frac{1}{2\cdot3^n}\r)^{\alpha}<x_2<\lf(\frac{1}{2\cdot3^n}\r)^{\alpha}\r\}\]
and
$$U^r_{n, 2^{n-1}}:=\lf\{x=(x_1,x_2)\in\Omega^+_{\psi^\alpha_c}:q_{n}^{2^{n-1}}<x_1<\fz\ {\rm and}\ \lf(\frac{1}{3}\r)^{\alpha}\lf(\frac{1}{2\cdot3^n}\r)^{\alpha}<x_2<\lf(\frac{1}{2\cdot3^n}\r)^{\alpha}\r\}.$$
On every $U^l_{n,k}$ with $k>1$, we define a function $v^+$ by setting 
\begin{equation}\label{eq:v1}
v^+(x):=\left\{\begin{array}{ll}
\frac{-x_2}{\lf(1-\lf(\frac{2}{3}\r)^{\alpha}\r)\lf(\frac{1}{2\cdot3^n}\r)^{\alpha}}+\frac{1}{\lf(1-\lf(\frac{2}{3^n}\r)^{\alpha}\r)}, &\ {\rm if}\ \lf(\frac{2}{3}\r)^{\alpha}\lf(\frac{1}{2\cdot3^n}\r)^{\alpha}<x_2<\lf(\frac{1}{2\cdot3^n}\r)^{\alpha},\\
1, &\ {\rm if} \lf(\frac{1}{2}\r)^{\alpha}\lf(\frac{1}{2\cdot3^n}\r)^{\alpha}\leq x_2\leq\lf(\frac{2}{3}\r)^{\alpha}\lf(\frac{1}{2\cdot3^n}\r)^{\alpha},\\
\frac{x_2}{\lf(\lf(\frac{1}{2}\r)^{\alpha}-\lf(\frac{1}{3}\r)^{\alpha}\r)\lf(\frac{1}{2\cdot3^n}\r)^{\alpha}}-\frac{\lf(\frac{1}{3}\r)^{\alpha}}{\lf(\frac{1}{2}\r)^{\alpha}-\lf(\frac{1}{3}\r)^{\alpha}}, &\ {\rm if}\ \lf(\frac{1}{3}\r)^{\alpha}\lf(\frac{1}{2\cdot3^n}\r)^{\alpha}<x_2<\lf(\frac{1}{2}\r)^{\alpha}\lf(\frac{1}{2\cdot3^n}\r)^{\alpha}.
\end{array}\right.
\end{equation}
On the set $\Omega^+_{\psi^\alpha_c}\setminus\bigcup_{n=1}^{\fz}\bigcup_{k=2}^{2^{n-1}}U^l_{n, k}$, we simply set $v^+(x)\equiv0$. 
For every $1<p<\fz$, we define a number $\alpha_p$ by setting 
\begin{equation}\label{eq:alpha}
\alpha_p^{-1}:=\left\{\begin{array}{ll}
\frac{((1+\alpha)-\frac{\log 2}{\log 3})p}{(1+\alpha)-\frac{\log 2}{\log 3}+(1-\alpha)p}, &\ {\rm if}\ \frac{(1+\alpha)-\frac{\log 2}{\log 3}}{2\alpha-\frac{\log 2}{\log 3}}<p<\fz,\\
1, &\ {\rm if}\ 1<p\leq \frac{(1+\alpha)-\frac{\log 2}{\log 3}}{2\alpha-\frac{\log 2}{\log 3}}.
\end{array}\right.
\end{equation} 
Then $1\leq\alpha^{-1}_p<p$.
Next, we define our test-function $u_e^+$ by setting 
\begin{equation}\label{eq:test}
u_e^+(x)=3^{n\alpha\beta}\lf(\frac{1}{n\log n}\r)^{\alpha_p}v^+(x)\ {\rm for\ every}\ x\in \Omega^+_{\psi^\alpha_c},
\end{equation}
with
\begin{equation}\label{equa:beta}
\beta\leq\frac{(1+\alpha)-\frac{\log 2}{\log 3}}{\alpha p}-1.
\end{equation}
First, a simple computation gives
\begin{eqnarray}
\int_{\mathbb H^+_{\psi^\alpha_c}}|u_e^+(x)|^pdx&\leq&C\sum_{n=1}^{\fz}\sum_{k=2}^{2^{n-1}}3^{n\alpha\beta p}\lf(\frac{1}{n\log n}\r)^{\alpha_pp}\mathcal H^2(U^l_{n, k})\\
                                                                   &\leq&C\sum_{n=1}^\fz2^n3^{n\lf(\beta\alpha p-\alpha-1\r)}\lf(\frac{1}{n\log n}\r)^{\alpha_pp}<\fz,\nonumber
\end{eqnarray}
for $2\cdot3^{(\beta\alpha p-\alpha-1)}<1$ and $\alpha_pp>1$. Furthermore
\begin{equation}\label{eq:deri2}
|Du_e^+(x)|\leq \left\{\begin{array}{ll}
C3^{{n\alpha(\beta+1)}}\lf(\frac{1}{n\log n}\r)^{\alpha_p}, &\ x\in U^l_{n, k}\ {\rm with}\ 1<k,\\
0, &\ {\rm elsewhere}.
\end{array}\right.
\end{equation}
Hence,
\begin{eqnarray}
\int_{\mathbb H^+_{\psi^\alpha_c}}|Du_e^+(x)|^pdx&\leq& C\sum_{n=1}^\fz\sum_{k=2}^{2^{n-1}}3^{n\alpha p(\beta+1)}\lf(\frac{1}{n\log n}\r)^{\alpha_pp}\mathcal H^2(U^l_{n, k})\\
                                                                     &\leq&C\sum_{n=1}^\fz2^n3^{n\lf(\alpha p(\beta+1)-\alpha-1\r)}\lf(\frac{1}{n\log n}\r)^{\alpha_pp}\nonumber\\
                                                                     &<&\fz,\nonumber
\end{eqnarray}
for $2\cdot 3^{\lf(\alpha p(\beta+1)-\alpha-1\r)}\leq 1$ and $\alpha_pp>1$. Hence, $u_e^+\in W^{1,p}(\Omega^+_{\psi^\alpha_c})$. For every $n\in\mathbb N\setminus\{1\}$ and $k\in\{2, 3,\cdots, 2^{n-1}\}$ and every 
$$\lf(\frac{1}{2}\r)^{\alpha}\lf(\frac{1}{2\cdot 3^n}\r)^{\alpha}\leq x_2\leq \lf(\frac{2}{3}\r)^{\alpha}\lf(\frac{1}{2\cdot3^n}\r)^{\alpha},$$
we define $S_{n, k}(x_2):=(\rr\times\{x_2\})\cap\mathcal C_{n, k}$ to be a horizontal line segment inside the cuspidal domain $\mathcal C_{n,k}$.
See the picture below. Assume that there exists an extension $E(u_e^+)\in W^{1,q}_{\rm loc}(\rr^2)$ for some $1\leq q\leq p$. By the $ACL$-characterization of Sobolev functions and the H\"older inequality, for almost every 
$$x_2\in\lf(\lf(\frac{1}{2}\r)^{\alpha}\lf(\frac{1}{2\cdot3^n}\r)^{\alpha}, \lf(\frac{2}{3}\r)^{\alpha}\lf(\frac{1}{2\cdot3^n}\r)^{\alpha}\r),$$ 
we have 
\begin{equation}\label{eq:inte5}
C\int_{S_{n,k}(x_2)}|DE(u_e^+)(x)|^qdx_1\geq3^{n\alpha\beta q+n(q-1)}\lf(\frac{1}{n\log n}\r)^{\alpha_pq},
\end{equation}
with a uniform constant $C$ independent of $n,k,x_2$. 

The following three propositions show the sharpness of the result in Theorem \ref{thm:E1}. 
\begin{prop}\label{thm:noex1'}
Let $0<\alpha\leq\frac{\log 2}{2\log 3}$ and let $\Gamma_{\psi^\alpha_c}$ be the corresponding Cantor-cuspidal graph. Then, for arbitrary $1<p\leq \fz$, the function $u_e^+$ cannot be extended to be a function in the class $W^{1,1}_{\rm loc}(\rr^2)$.
\end{prop}
\begin{proof}
Since $0<\alpha\leq\frac{\log 2}{2\log 3}$, we can choose $1-\frac{\log 2}{\alpha\log 3}\leq\beta\leq-1$ in the definition of the function $u^+_e$ in (\ref{eq:test}). Then $u_e^+\in W^{1, \fz}(\mathbb H^+_{\psi^\alpha_c})$. Since both $|u_e^+|$ and $|Du_e^+|$ vanish outside a bounded set, the H\"older inequality implies $u_e^+\in W^{1,p}(\Omega^+_{\psi^\alpha_c})$, for every $1<p\leq\fz$. Assume that there exists an extension  $E(u_e^+)$ in the class $W^{1,1}_{\rm loc}(\rr^2)$. By (\ref{eq:inte5}) and the Fubini theorem, we obtain 
\begin{eqnarray}\label{eq:E11'}
\int_{D(0, 2)}|D E(u_e^+)(x)|dx&\geq& C\sum_{n=1}^{\fz}\sum_{k=2}^{2^{n-1}}3^{n\alpha(\beta-1)}\lf(\frac{1}{n\log n}\r)^{\alpha_p}\\
                                             &\geq& C\sum_{n=1}^\fz2^n3^{n\alpha(\beta-1)}\lf(\frac{1}{n\log n}\r)^{\alpha_p}=\fz.\nonumber
\end{eqnarray}
This contradicts  the assumption that $E(u^+_e)\in W^{1,1}_{\rm loc}(\rr^2)$ and the proof is finished.  
\end{proof}
\begin{prop}\label{thm:noex2'}
Let $\frac{\log 2}{2\log 3}<\alpha<1$ and $\Gamma_{\psi^\alpha_c}$ be the corresponding Cantor-cuspidal graph. Then for arbitrary $1< p\leq\frac{(1+\alpha)-\frac{\log 2}{\log 3}}{2\alpha-\frac{\log 2}{\log 3}}$, there exists a function $u_e^+\in W^{1,p}(\Omega^+_{\psi^\alpha_c})$ which can not be extended to be a function in the class $W^{1,1}_{\rm loc}(\rr^2)$. 
\end{prop}
\begin{proof}
For every $1< p\leq\frac{(1+\alpha)-\frac{\log 2}{\log 3}}{2\alpha-\frac{\log 2}{\log 3}}$, we fix $\beta=\frac{(1+\alpha)-\frac{\log 2}{\log 3}}{\alpha p}-1$ in the definition of the function $u_e^+\in W^{1,p}(\Omega^+_{\psi^\alpha_c})$ in (\ref{eq:test}). Assume that there exists an extension  $E(u^+_e)\in W^{1,1}_{\rm loc}(\rr^2)$. Then by (\ref{eq:inte5}) and the Fubini theorem, we have 
\begin{eqnarray}\label{eq:E12'}
\int_{D(0, 2)}|D E(u_e^+)(x)|dx&\geq& C\sum_{n=1}^\fz\sum_{k=2}^{2^{n-1}}3^{n\alpha(\beta-1)}\lf(\frac{1}{n\log n}\r)^{\alpha_p}\\
                                             &\geq& C\sum_{n=1}^\fz2^n3^{n\alpha(\beta-1)}\lf(\frac{1}{n\log n}\r)^{\alpha_p}=\fz,\nonumber
\end{eqnarray}
since $2\cdot 3^{n\alpha(\beta-1)}\geq 1$ and $\alpha_p=1$ for $1<p\leq\frac{(1+\alpha)-\frac{\log 2}{\log 3}}{2\alpha-\frac{\log 2}{\log 3}}$. This contradicts the assumption that $E(u_e^+)\in W^{1,1}_{\rm loc}(\rr^2)$ and the proof is finished.
\end{proof}

\begin{prop}\label{thmnoex3}
Let $\frac{\log 2}{2\log 3}<\alpha<1$ and $\Gamma_{\psi^\alpha_c}$ be the corresponding Cantor-cuspidal graph. For arbitrary $\frac{(1+\alpha)-\frac{\log 2}{\log 3}}{2\alpha-\frac{\log 2}{\log 3}}<p<\fz$, there exists a function $u_e^+\in W^{1,p}(\Omega^+_{\psi^\alpha_c})$ which can not be extended to be a function in the class $W^{1,q}_{\rm loc}(\rr^2)$, whenever $\frac{((1+\alpha)-\frac{\log 2}{\log 3})p}{(1+\alpha)-\frac{\log 2}{\log 3}+(1-\alpha)p}\leq q<\fz$.
\end{prop}
\begin{proof}
Define
$$q_o=\frac{((1+\alpha)-\frac{\log 2}{\log 3})p}{(1+\alpha)-\frac{\log 2}{\log 3}+(1-\alpha)p}.$$
The H\"older inequality implies that it suffices to show that $u^+_e$ can not be extended to be a function in the class $W^{1,q_o}_{\rm loc}(\rr^2)$.
Fix $\beta=\frac{(1+\alpha)-\frac{\log 2}{\log 3}}{\alpha p}-1$ in the definition of $u^+_e$ in (\ref{eq:test}). By (\ref{eq:inte5}) and the Fubini theorem, we have 
\begin{eqnarray}\label{eq:E13'}
\int_{D(0, 2)}|DE(u_e^+)(x)|^{q_o}dx&\geq& C\sum_{n=1}^\fz\sum_{k=2}^{2^{n-1}}3^{(n\beta q_o-n)\alpha+n(q_o-1)}\lf(\frac{1}{n\log n}\r)^{\alpha_pq_o}\\
                                        &\geq&C\sum_{n=1}^\fz2^n3^{\alpha(n\beta q_o-n)+n(q_o-1)}\lf(\frac{1}{n\log n}\r)^{\alpha_pq_o}=\fz, \nonumber
\end{eqnarray}
since $2\cdot 3^{\alpha(\beta q_o-1)+(q_o-1)}=1$ and $\alpha_pq_o=1$. This contradicts  the assumption that $E(u_e^+)\in W^{1,q_o}_{\rm loc}(\rr^2)$ and the proof is finished.
\end{proof}

\section{Sobolev extendability for $\Omega^-_{\psi^\alpha_c}$}
\subsection{Extension from $W^{1,p}(\Omega^-_{\psi^\alpha_c})$ to $W^{1,q}_{\rm loc}(\rr^2)$}
\begin{thm}\label{thm:E2}
Let $\Gamma_{\psi^\alpha_c}\subset\rr^2$ be a Cantor-cuspidal graph with $\frac{\log 2}{2\log 3}<\alpha<1$. Then the reflection $\mr:\rr^2\to\rr^2$ over $\Gamma_{\psi^\alpha_c}$ defined in (\ref{equa:reflec}) and (\ref{equa:reflec1}) induces a bounded linear extension operator from $W^{1,p}(\Omega^-_{\psi^\alpha_c})$ to $W^{1,q}_{\rm loc}(\rr^2)$, whenever $\frac{(1+\alpha)-\frac{\log 2}{\log 3}}{2\alpha-\frac{\log 2}{\log 3}}<p<\fz$ and $1\leq q<\frac{((1+\alpha)-\frac{\log 2}{\log 3})p}{(1+\alpha)-\frac{\log 2}{\log 3}+(1-\alpha)p}$.  
\end{thm}
\begin{proof}
Simply replace $\mathsf R^-_{n, k}$ by $\mathsf R^+_{n, k}$ in the proof of Theorem \ref{thm:E1} and repeat the argument. 
\end{proof}

\subsection{Sharpness of Theorem \ref{thm:E2}}
Let us define a function $v_-$ on $\Omega^-_{\psi^\alpha_c}$. For every even $n\in\mathbb N$ and $k=1, 2, ..., 2^{n-1}$, we define 
\begin{equation}\label{equa:even}
v_-(x)=0, \ {\rm for\ every}\ x\in\mathcal C_{n,k}.\nonumber
\end{equation}
For odd $n\in\mathbb N$ and every $k=1, 2, ..., 2^{n-1}$, we define the function $v_-$ on $\mathcal C_{n, k}$ by setting
\begin{equation}\label{eq:E9}
v_-(x):=\left\{\begin{array}{ll}
1,&\  x_2>\lf(\frac{1}{6}\r)^{\alpha}\lf(\frac{1}{2\cdot3^n}\r)^{\alpha},\\
\frac{x_2-\lf(\frac{1}{9}\r)^{\alpha}\lf(\frac{1}{2\cdot3^n}\r)^{\alpha}}{\lf(\lf(\frac{1}{6}\r)^{\alpha}-\lf(\frac{1}{9}\r)^{\alpha}\r)\lf(\frac{1}{2\cdot3^n}\r)^{\alpha}}, &\  \lf(\frac{1}{9}\r)^{\alpha}\lf(\frac{1}{2\cdot3^n}\r)^{\alpha}\leq x_2\leq \lf(\frac{1}{6}\r)^{\alpha}\lf(\frac{1}{2\cdot3^n}\r)^{\alpha},\\
0,&\  x_2<\lf(\frac{1}{9}\r)^{\alpha}\lf(\frac{1}{2\cdot3^n}\r)^{\alpha}.
\end{array}\right.
\end{equation}
Outside the set $\cup_{n=1}^{\fz}\cup_{k=1}^{2^{n-1}}\mathcal C_{n,k}$, we just set $v_-=0$.  Finally, we define our function $u_e^-$ on $\Omega^-_{\psi^\alpha_c}$ by setting
\begin{equation}\label{equa:func}
u^-_e(x)=3^{n\alpha\beta}\lf(\frac{1}{n\log n}\r)^{\alpha_p}v_-(x),
\end{equation}
where $\alpha_p$ is given in (\ref{eq:alpha}) and $\beta$ is given in (\ref{equa:beta}). Then we have 
\begin{eqnarray}\label{eq:E10}
\int_{\mathbb H^-_{\psi^\alpha_c}}|u^-_e(x)|^pdx&\leq&\sum_{n\ {\rm is\ odd}}\sum_{k=1}^{2^{n-1}}\int_{\mathcal C_{n,k}}|u^-_e(x)|^pdx\\
&\leq& C\sum_{n\ {\rm is\ odd}}2^n3^{n\lf(p\alpha\beta-1-\alpha\r)}\lf(\frac{1}{n\log n}\r)^{\alpha_pp}<\fz,\nonumber
\end{eqnarray}
since $2\cdot3^{p\alpha\beta-1-\alpha}<1$. There exists a positive constant $C>0$ such that for odd $n$
\begin{equation}\label{equa:lessC}
|D u^-_e(x)|\leq
C3^{n\alpha(\beta+1)}\lf(\frac{1}{n\log n}\r)^{\alpha_p},
\end{equation}
for every $x\in\mathcal C_{n, k}$ with $\lf(\frac{1}{9}\r)^{\alpha}\lf(\frac{1}{2\cdot 3^n}\r)^{\alpha}\leq x_2\leq\lf(\frac{1}{6}\r)^{\alpha}\lf(\frac{1}{2\cdot 3^n}\r)^{\alpha}$. Moreover $Du^-_e(x)=0$ elsewhere. Hence, we have 
\begin{eqnarray}
\int_{\mathbb H^-_{\psi^\alpha_c}}|D u^-_e(x)|^pdx&\leq& C\sum_{n\ {\rm is\ odd}}\sum_{k=1}^{2^{n-1}}3^{n\lf(p\alpha(\beta+1)-(1+\alpha)\r)}\lf(\frac{1}{n\log n}\r)^{\alpha_pp}\\
                                          &\leq&C\sum_{n\ {\rm is\ odd}}2^n3^{n\lf(p\alpha(\beta+1)-(1+\alpha)\r)}\lf(\frac{1}{n\log n}\r)^{\alpha_pp}<\fz,\nonumber
\end{eqnarray}
since $2\cdot3^{p\alpha(\beta+1)-(1+\alpha)}\leq 1$ and $\alpha_pp>1$. Hence, $u^-_e\in W^{1,p}(\Omega^-_{\psi^\alpha_c})$. Fixing an odd $n\in\mathbb N$ and $k=1, 2, ..., 2^{n-1}$, there exists two cuspidal domains $\mathcal C_{n+1, k_1}$ and $\mathcal C_{n+1, k_1'}$ nearby from the next generation. One is on the right-hand side of $\mathcal C_{n, k}$ and the other is on the left-hand side of $\mathcal C_{n, k}$. Without loss of generalization, we assume $\mathcal C_{n+1,k_1}$ is on the left-hand side of $\mathcal C_{n,k}$ and $\mathcal C_{n+1,k_2}$ is on the right-hand side of $\mathcal C_{n, k}$. For every 
\begin{equation}
\lf(\frac{1}{6}\r)^{\alpha}\lf(\frac{1}{2\cdot3^n}\r)^{\alpha}<x_2< \lf(\frac{1}{3}\r)^{\alpha}\lf(\frac{1}{2\cdot3^n}\r)^{\alpha}\nonumber
\end{equation}
we define
$$L^l_{n, k}:=\lf\{x=(x_1,x_2)\in\mathbb H^+_{\psi^\alpha_c}:q_{n+1}^{k_1}<x_1<q_n^k\r\}$$
and
$$L^r_{n, k}:=\lf\{x=(x_1,x_2)\in\mathbb H^+_{\psi^\alpha_c}: q_n^k<x_1<q_{n+1}^{k_2}\r\}.$$
Assume that there exists an extension $E(u^-_e)\in W^{1,q}_{\rm loc}(\rr^2)$. By the $ACL$-characterization of Sobolev functions and the H\"older inequality, we have 
\begin{equation}\label{equa:inte1}
C\int_{L^l_{n ,k}(x_2)}|D E(u^-_e)(x)|^qdx_1\geq3^{n\alpha\beta q+n(q-1)}\lf(\frac{1}{n\log n}\r)^{\alpha_pq}
\end{equation}
and
\begin{equation}\label{equa:inte2}
 C\int_{L^r_{n ,k}(x_2)}|D E(u^-_e)(x)|^qdx_1\geq3^{n\alpha\beta q+n(q-1)}\lf(\frac{1}{n\log n}\r)^{\alpha_pq}
\end{equation}
for almost every 
\begin{equation}
x_2\in\lf(\lf(\frac{1}{6}\r)^{\alpha}\lf(\frac{1}{2\cdot3^n}\r)^{\alpha}, \lf(\frac{1}{3}\r)^{\alpha}\lf(\frac{1}{2\cdot3^n}\r)^{\alpha}\r),\nonumber
\end{equation}
where the constant $C$ is independent of $n, k, x_2$. 

The following propositions show the sharpness of the result in Theorem \ref{thm:E2}.
\begin{prop}\label{thm:noex1}
Let $0<\alpha\leq\frac{\log 2}{2\log 3}$ and $\Gamma_{\psi^\alpha_c}$ be the corrersponding Cantor-cuspidal graph. Then for arbitrary $1< p\leq \fz$, there exists a function $u^-_e\in W^{1,p}(\Omega^-_{\psi^\alpha_c})$ which cannot be extended to be a function in the class $W^{1,1}_{\rm loc}(\rr^2)$.
\end{prop}
\begin{proof}
Since $0<\alpha\leq\frac{\log 2}{2\log 3}$, we fix $1-\frac{\log 2}{\alpha\log 3}\leq\beta\leq-1$ in the definition of $u^-_e$ in (\ref{equa:func}). Then $u^-_e\in W^{1, \fz}(\mathbb H^-_{\psi^\alpha_c})$. Since both $|u^-_e|$ and $|Du^-_e|$ vanish outside a bounded set, the H\"older inequality implies $u^-_e\in W^{1,p}(\mathbb H^-_{\psi^\alpha_c})$, for every $1<p\leq\fz$. Assume that there exists an  function $E(u^-_e)$ in the class $W^{1,1}_{\rm loc}(\rr^2)$. By (\ref{equa:inte1}), (\ref{equa:inte2}) and the Fubini theorem, we have
\begin{eqnarray}\label{eq:E11}
\int_{D(0, 2)}|D E(u^-_e)(x)|dx&\geq& C\sum_{n\ {\rm is\ odd}}\sum_{k=1}^{2^{n-1}}3^{n\alpha(\beta-1)}\lf(\frac{1}{n\log n}\r)^{\alpha_p}\\
                                             &\geq& C\sum_{n\ {\rm is\ odd}}2^n3^{n\alpha(\beta-1)}\lf(\frac{1}{n\log n}\r)^{\alpha_p}=\fz,\nonumber
\end{eqnarray}
since $2\cdot 3^{\alpha(\beta-1)}\geq 1$ and $\alpha_p\leq 1$ This contradicts the assumption that $E(u^-_e)\in W^{1,1}_{\rm loc}(\rr^2)$ and the proof is finished.
\end{proof}

\begin{prop}\label{thm:noex2}
Let $\frac{\log 2}{2\log 3}<\alpha<1$ and $\Gamma_{\psi^\alpha_c}$ be the corresponding Cantor-cuspidal graph. Then, for arbitrary $1<p\leq\frac{(1+\alpha)-\frac{\log 2}{\log 3}}{2\alpha-\frac{\log 2}{\log 3}}$, there exists a function $u^-_e\in W^{1,p}(\Omega^-_{\psi^\alpha_c})$ which can not be extended to be a function in the class $W^{1,1}_{\rm loc}(\rr^2)$. 
\end{prop}
\begin{proof}
For every $1<p\leq\frac{(1+\alpha)-\frac{\log 2}{\log 3}}{2\alpha-\frac{\log 2}{\log 3}}$, we fix $\beta=\frac{(1+\alpha)-\frac{\log 2}{\log 3}}{\alpha p}-1$ in the definition of the function  $u^-_e\in W^{1,p}(\Omega^-_{\psi^\alpha_c})$ in (\ref{equa:func}). Assume that there exists an extended function $E(u^-_e)$ in the class $W^{1,1}_{\rm loc}(\rr^2)$. Then by (\ref{equa:inte1}), (\ref{equa:inte2}) and the Fubini theorem, we have 
\begin{eqnarray}\label{eq:E12}
\int_{D(0, 2)}|D E(u^-_e)(x)|dx&\geq& C\sum_{n\ {\rm is\ odd}}\sum_{k=1}^{2^{n-1}}3^{{n\alpha(\beta-1)}}\lf(\frac{1}{n\log n}\r)^{\alpha_p}\\
                                             &\geq& C\sum_{n\ {\rm is\ odd}}2^n3^{{n\alpha(\beta-1)}}\lf(\frac{1}{n\log n}\r)^{\alpha_p}=\fz,\nonumber
\end{eqnarray}
since $2\cdot3^{\alpha(\beta-1)}\geq 1$ and $\alpha_p\leq 1$. This contradicts the assumption that $E(u^-_e)\in W^{1,1}_{\rm loc}(\rr^2)$ and the proof is finished.
\end{proof}

\begin{prop}\label{thmnoex3}
Let $\frac{\log 2}{2\log 3}<\alpha<1$ and $\Gamma_{\psi^\alpha_c}$ be the corresponding Cantor-cuspidal graph. For arbitrary $\frac{(1+\alpha)-\frac{\log 2}{\log 3}}{2\alpha-\frac{\log 2}{\log 3}}<p<\fz$, there exists a function $u^-_e\in W^{1,p}(\Omega^-_{\psi^\alpha_c})$ which can not be extended to be a function in the class $W^{1,q}_{\rm loc}(\rr^2)$, whenever $\frac{((1+\alpha)-\frac{\log 2}{\log 3})p}{(1+\alpha)-\frac{\log 2}{\log 3}+(1-\alpha)p}\leq q<\fz$.
\end{prop}
\begin{proof}
Fix 
$$q_o=\frac{((1+\alpha)-\frac{\log 2}{\log 3})p}{(1+\alpha)-\frac{\log 2}{\log 3}+(1-\alpha)p}.$$
The H\"older inequality implies that it suffices to show that $u^-_e$ can not be extended to be a function in the class $W^{1,q_o}_{\rm loc}(\rr^2)$.
Fix $\beta=\frac{(1+\alpha)-\frac{\log 2}{\log 3}}{\alpha p}-1$ in the definition of $u^-_e$ in (\ref{equa:func}). Assume that there exists an extension  $E(u^-_e)\in W^{1,q_o}_{\rm loc}(\rr^2)$. By (\ref{equa:inte1}), (\ref{equa:inte2}) and the Fubini theorem, we have 
\begin{eqnarray}\label{eq:E13}
\int_{D(0, 2)}|DE(u^-_e)(x)|^{q_o}dx&\geq& C\sum_{n\ {\rm is\ odd}}\sum_{k=1}^{2^{n-1}}3^{{(n\beta q_o-n})\alpha+n(q_o-1)}\lf(\frac{1}{n\log n}\r)^{\alpha_pq_o}\\
                                        &\geq&C\sum_{n\ {\rm is\ odd}}2^n3^{\alpha({n\beta q_o-n})+n(q_o-1)}\lf(\frac{1}{n\log n}\r)^{\alpha_pq_o}=\fz,\nonumber
\end{eqnarray}
since $2\cdot 3^{\alpha(\beta q_o-1)+(q_o-1)}=1$ and $\alpha_pq_o=1$. This contradicts the assumption that $E(u^-_e)\in W^{1,q_o}_{\rm loc}(\rr^2)$ and the proof is finished.
\end{proof}

\end{document}